\documentclass[english]{amsart}

\usepackage{esint}
\usepackage[svgnames]{xcolor}
\usepackage[colorlinks,citecolor=red,pagebackref,hypertexnames=false,breaklinks]{hyperref}
\usepackage{pgf,tikz}
\usepackage{pdfsync}

\usepackage{dsfont}
\usepackage{url}
\usepackage[utf8]{inputenc}
\usepackage[T1]{fontenc}
\usepackage{lmodern}
\usepackage{babel}
\usepackage{mathtools}
\usepackage{amssymb}
\usepackage{lipsum}
\usepackage{mathrsfs}
\usepackage{color}

\newtheorem{theorem}{Theorem}[section]
\newtheorem{proposition}{Proposition}[section]
\newtheorem{lemma}{Lemma}[section]

\newtheorem{corollary}{Corollary}[section]

\newtheorem{remark}{Remark}[section]

\numberwithin{equation}{section}

\title[Semilinear elliptic inverse problem]{Stability estimate for a  semilinear elliptic inverse problem}

\author[Mourad Choulli]{Mourad Choulli}
\address{Universit\'e de Lorraine, 4 cours L\'eopold, 54052 Nancy cedex, France}
\email{mourad.choulli@univ-lorraine.fr}

\author[Guanghui Hu]{Guanghui Hu}
\address{School of Mathematical Sciences, Nankai University, Tianjing 300071, China}
\email{ghhu@nankai.edu.cn}

\author[Masahiro Yamamoto]{Masahiro Yamamoto}
\address{Graduate School of Mathematical Sciences, The University of Tokyo, Komaba, Meguro, Tokyo 153-8914, Japan
\\
Honorary Member of Academy of Romanian Scientists, Splaiul Independentei Street, no 54, 050094, Bucharest Romania
\\
People's Friendship University of Russia (RUDN University) 6 Miklukho-Maklaya St, Moscow, 117198, Russian Federation}
\email{myama@ms.u-tokyo.ac.jp}

\thanks{MC is supported by the grant ANR-17-CE40-0029 of the French National Research Agency ANR (project MultiOnde). 
MY is supported by Grant-in-Aid for Scientific Research (S) 15H05740 of Japan Society for the Promotion of Science and by The National Natural Science Foundation of China (no. 11771270, 91730303).   This work was supported by A3 Foresight Program``Modeling and Computation of  Applied Inverse Problems'' of Japan Society for the Promotion of Science and prepared with the support of the "RUDN University Program 5-100".}

\date{}

\begin{document}

\begin{abstract}
We establish a logarithmic stability estimate for the inverse problem of determining the nonlinear term, appearing in a semilinear boundary value problem, 
from the corresponding Dirichlet-to-Neumann map. Our result can be seen as a stability inequality for an earlier uniqueness result by Isakov and Sylvester [Commun. Pure Appl. Math. 47 (1994), 1403-1410].
\end{abstract}

\subjclass[2010]{35R30}

\keywords{Semilinear elliptic BVP, Dirichlet-to-Neumann map, stability inequality}

\maketitle

\tableofcontents

\section{Introduction}\label{sec:1}

Let $\Omega$ be a $C^{1,1}$ bounded domain of $\mathbb{R}^n$ ($n\geq 2$) with boundary $\Gamma$. Fix $\mathfrak{c}=(c_0,c_1,c)$ with $c_0>0$, $c_1>0$ and $0\le c<\lambda_1(\Omega)$, where $\lambda_1(\Omega)$ denotes the first eigenvalue of the Laplace operator on $\Omega$ with Dirichlet boundary condition. 

We denote by $\mathscr{A}(\mathfrak{c},\alpha )$, with $\alpha \ge 0$, the set of continuously differentiable functions $a:\mathbb{R}\rightarrow \mathbb{R}$ satisfying the following two assumptions
\begin{equation}\label{a1}
|a(t)|\le c_0 +c_1 |t|^\alpha\quad \mbox{for all $t\in \mathbb{R}$},
\end{equation}
and
\begin{equation}\label{a2}
a'(t)\ge -c \quad \mbox{for all $t\in \mathbb{R}$}.
\end{equation}

In the present article, the ball of a normed space $X$ at center $0$ with 
radius $M>0$ is denoted by $B_X(M)$. Also, $c_\Omega$ denotes 
a generic constant only depending on $\Omega$.

Unless otherwise stated all the functions we use are assumed to real-valued.

Consider the following non-homogenous semilinear boundary value problem
\begin{equation}\label{2.1}
\left\{
\begin{array}{ll}
-\Delta u+a\circ u=0\quad &\mbox{in}\; \Omega ,
\\
u=f &\mbox{on}\; \Gamma.
\end{array}
\right.
\end{equation}

Henceforth we use the abbreviation BVP for boundary value problem.
For the formulation of our inverse problem, we need the well-posedness
of the BVP \eqref{2.1}, which is stated as follows:

\begin{theorem}\label{theorem1.1}
 Assume that $\alpha$ is arbitrary if $n=2$ and $\alpha\le n/(n-2)$ if $n\ge 3$. Let $a\in \mathscr{A}(\mathfrak{c},\alpha)$. Then, for any $f\in H^{3/2}(\Gamma )$,  the BVP \eqref{2.1} has a unique solution $u_a(f)\in H^2(\Omega )$. Furthermore,
 \begin{equation}\label{1.1}
\|u_a(f)\|_{H^2(\Omega )}\le C,\quad \mbox{for any}\; f\in B_{H^{3/2}(\Gamma )}(M),
\end{equation}
where $C=C(\Omega, M,\mathfrak{c}, \alpha)>0$ is a constant. That is $f \rightarrow u_a(f)$ maps bounded set of $H^{3/2}(\Gamma )$ into bounded set of $H^2(\Omega)$.
\end{theorem}

An example of a function $a$ fulfilling the assumptions in the above theorem is the linear case $a(t)=-kt$ with $k<c$, which models the time-harmonic acoustic wave propagation at the wavenumber $k>0$. The semilinear equation also covers the Schr\"odinger equation.

Hereafter, the derivative in the direction of the unit exterior normal vector field $\nu$ on $\Gamma$ of a function $u$ is denoted by $\partial_\nu u$.

\begin{theorem}\label{theorem1.2}
(i) Assume that $\alpha$ is arbitrary if $n=2$ and $\alpha\le 3$ if $n=3$. If $a\in \mathscr{A}(\mathfrak{c},\alpha )$ then we can define the mapping
\[
\Lambda_a:H^{3/2}(\Gamma )\rightarrow H^{1/2}(\Gamma ):f\mapsto \partial_\nu u_a(f).
\]
Moreover, for arbitrarily given $M>0$, we have
\begin{equation}\label{1.2}
\|\Lambda_a(f)\|_{H^{1/2}(\Gamma )} \le C,\quad \mbox{for any}\; f\in B_{H^{3/2}(\Gamma )}(M),
\end{equation}
where $C=C(\Omega, M, \mathfrak{c}, \alpha)$ is a constant.
\\
(ii) Assume that $n>4$. Let $n/2<p<n$ and $\alpha \le q/p$ with $q=2n/(n-4)$. If $a\in \mathscr{A}(\mathfrak{c},\alpha)$ then we can define
\[
\Lambda_a:W^{2-1/p,p}(\Gamma )\rightarrow \partial_\nu u_a(f)\in W^{1-1/p,p}(\Gamma ):f\mapsto \partial_\nu u_a(f).
\]
Furthermore, for  arbitrarily given $M>0$, we have
\begin{equation}\label{1.3}
\|\Lambda_a(f)\|_{W^{1-1/p,p}(\Gamma )} \le C,\quad \mbox{for any}\; f\in B_{W^{2-1/p,p}(\Gamma )}(M).
\end{equation}
Here $C=C(\Omega, M,\mathfrak{c} ,p ,\alpha )>0$ is a constant.
\\
(iii) Assume that $n=4$. Let $2<p<4$, $1\le r<2$, $q=2r/(2-r)$ and $\alpha \le q/p$. If $a\in \mathscr{A}(\mathfrak{c},\alpha)$ then we can define
\[
\Lambda_a:W^{2-1/p,p}(\Gamma )\rightarrow  W^{1-1/p,p}(\Gamma ):f\mapsto \partial_\nu u_a(f).
\]
Moreover, for any $M>0$, we have
\begin{equation}\label{1.4}
\|\Lambda_a(f)\|_{W^{1-1/p,p}(\Gamma )} \le C,\quad \mbox{for any}\; f\in B_{W^{2-1/p,p}(\Gamma )}(M),
\end{equation}
where $C=C(\Omega , M,\mathfrak{c} , p, r ,\alpha)>0$ is a constant.
\end{theorem}

We call the (nonlinear) operator $\Lambda_a$ in Theorem \ref{theorem1.2} the Dirichlet-to-Neumann map associated to $a$.

We are concerned with the inverse problem of determining the nonlinear term $a$ from the corresponding Dirichlet-to-Neumann map $\Lambda_a$. The main purpose is the stability issue.

For most of inverse problems, the solutions of the inverse problem 
do not necessarily depend on data 
continuously by conventional choices of topologies even if the uniqueness holds.
It is often that if we suitably reduce an admissible set of unknowns, 
then we can recover the stability for the inverse problem. 

Thus we define $\tilde{\mathscr{A}}(\mathfrak{c},
\alpha )$ as an admissible set of functions $a\in \mathscr{A}(\mathfrak{c},\alpha)$ satisfying the additional condition: for any $R>0$, there exists a constant $\varkappa_R$ so that
\begin{equation}\label{1.5}
|a'(u)-a'(v)|\le \varkappa_R|u-v|,\quad  |u|,\; |v| \le R.
\end{equation}

Note that condition \eqref{1.5} means that the first derivative of $a$ is Lipschitz continuous on bounded sets of $\mathbb{R}$. Also, we observe that the constant $\varkappa_R$ in \eqref{1.5} may depend on $a$.

Within this class, we can linearize the inverse problem under consideration. Precisely, we have the following proposition in which, for $j=0,1$,
\[
\mathscr{X}_j=H^{3/2-j}(\Gamma )\; \mbox{if}\; n=2,3\quad\mbox{and}\quad \mathscr{X}_j=W^{2-j-1/p,p}(\Gamma )\; \mbox{if}\; n\ge 4,
\]
and the space
\[
\mathscr{Y}=\mathscr{B}(\mathscr{X}_0, \mathscr{X}_1)
\]
denotes the set of bounded linear operators mapping $\mathscr{X}_0$ into $\mathscr{X}_1$.

The proposition below states that the linearization of the Dirichlet-to-Neumann map $\Lambda_a$ is the Dirichlet-to-Neumann map of the linearized problem.

\begin{proposition}\label{proposition1.1}
Under the assumptions and the notations of Theorem \ref{theorem1.2}, 
if $a\in \tilde{\mathscr{A}}(\mathfrak{c},\alpha )$, then $\Lambda_a$ 
is Fr\'echet differentiable at  any $f\in \mathscr{X}_0$ with $\Lambda'_a(f)(h)=\partial_\nu v_{a,f}(h)$, where $h\in \mathscr{X}_0$ and $v_{a,f}(h)$ is the unique solution of the BVP
\[
\left\{
\begin{array}{ll}
-\Delta v+a'\circ u_a(f) v=0\quad &\rm{in}\; \Omega ,
\\
v=h &\rm{on}\; \Gamma.
\end{array}
\right.
\]
Moreover, for any $M>0$, we have
\[
\|\Lambda_a '(f)\|_{\mathscr{Y}}\le C,\quad \mbox{for any}\; f\in B_{\mathscr{X}_0}(M).
\]
Here the constant $C>0$ is as Theorem \ref{theorem1.2}.
\end{proposition}

Henceforward $|\Gamma|$ denotes  the Lebesgue measure of $\Gamma$.

The main result of this paper is the following theorem.

\begin{theorem}\label{theorem1.3}
Assume that $n\ge 3$ and the assumptions of Theorem \ref{theorem1.2} hold for $a,\tilde{a}\in \tilde{\mathscr{A}}(\mathfrak{c},\alpha)$ satisfying $a(0)=\tilde{a}(0)$ and let $\beta=1/2$ if $n=3$ and $\beta=2-n/p$ if $n\ge 4$. Let $0<s<\min(1/2,\beta)$. Then
\[
\max_{|\lambda |\le M}|a(\lambda )-\tilde{a}(\lambda )|\le C_M\Psi\left(\sup_{\|f\|_{\mathscr{X}_0}\le \sqrt{|\Gamma|}M} \|\Lambda'_a(f)-\Lambda'_{\tilde{a}}(f)\|_{\mathscr{Y}}\right),
\]
where the constant $C_M=C$ is as in Theorem \ref{theorem1.2},
and
\[
\Psi (t)=\left\{
\begin{array}{ll}
|\ln t|^{-[2\min (1/2,s/n)\beta]/(n+2\beta)}+t\quad &\mbox{if}\; t>0,
\\
0 &\mbox{if}\; t=0.
\end{array}
\right.
\]
\end{theorem}

Theorem \ref{theorem1.3} immediately yields

\begin{corollary}\label{corollary1.2}
If $a,\tilde{a}\in \tilde{\mathscr{A}}(\mathfrak{c},\alpha)$ satisfy $a(0)=\tilde{a}(0)$ and $\Lambda_a =\Lambda_{\tilde{a}}$ then $a=\tilde{a}$.
\end{corollary}

This corollary corresponds to the uniqueness result in \cite{IS}
which considers more general equations $-\Delta u + a(x,u(x)) = 0$.

\begin{remark}\label{remark1.1}
{\rm
(a) Consider the Fr\'echet space $C(\mathbb{R})$ equipped with the family of semi-norms $(\mathfrak{p}_j)_{j\ge 1}$:
\[
\mathfrak{p}_j(h)=\max_{|t|\le j}|h(t)|,\quad h\in C(\mathbb{R}).
\]
Let $C_{\rm loc}^1(\mathscr{X}_0,\mathscr{X}_1)$ be the vector space of Fr\'echet differentiable functions 
\[
\Lambda :\mathscr{X}_0\rightarrow \mathscr{X}_1\]
so that $\Lambda$ and $\Lambda'$ are locally bounded. A natural topology on $C_{\rm loc}^1(\mathscr{X}_0,\mathscr{X}_1)$ is induced by the following family of semi-norms
\[
\mathfrak{q}_j(\Lambda )=\sup_{f\in B_{\mathscr{X}_0}(j|\Gamma|)}\left(\|\Lambda (f)\|_{\mathscr{X}_1}+\|\Lambda '(f)\|_{\mathscr{Y}}\right),\quad \Lambda \in C_{\rm loc}^1(\mathscr{X}_0,\mathscr{X}_1).
\]
We observe that the estimate in Theorem \ref{theorem1.3} can be rewritten in the form
\[
\mathfrak{p}_j(a-\tilde{a})\le C_j\Psi (\mathfrak{q}_j(\Lambda_a -\Lambda_{\tilde{a}})),\quad j\ge 1.
\]
(b) A natural distance on $\tilde{\mathscr{A}}(\mathfrak{c},\alpha)$ is given by
\[
\mathbf{d}(a,\tilde{a})=\sup_{|t|\le 1}|a(t)-\tilde{a}(t)|+\sup_{|t|\ge 1}|t^{-\alpha}(a(t)-\tilde{a}(t))|,\quad a,\tilde{a}\in \tilde{\mathscr{A}}(\mathfrak{c},\alpha).
\]
One can then ask whether it is possible to prove a stability estimate when $\tilde{\mathscr{A}}(\mathfrak{c},\alpha)$ is endowed with distance $\mathbf{d}$. They are two obstructions to get such kind of estimate. The first obstruction is due to the fact the natural space of Dirichlet-to-Neumann maps (defined in (a)) is  a locally convex metrizable topological vector space which is not normable. The second obstruction comes from the fact the local modulus of continuity in Theorem \ref{theorem1.3} is logarithmic.
}
\end{remark}

It is worth mentioning that the proof of Theorem \ref{theorem1.3} can be adapted to a partial Dirichlet-to-Neumann map of $\Lambda_a$.
Here, with fixed compact subsets $\Gamma', \Gamma''$ of $\Gamma$,
a partial Dirichlet-to-Neumann map means a mapping 
\[
f \in \{h\in H^{3/2}(\Gamma);\; \mbox{supp}(h)\subset \Gamma '\} \rightarrow \partial_{\nu}u_a(f)\vert_{\Gamma''} \in H^{1/2}(\Gamma'').
\]

A double logarithmic stability inequality for the linearized problem, with a partial Dirichlet-to-Neumann map, was recently established by Caro, Dos Santos Ferreira and Ruiz \cite{CDR}. One can expect by \cite{CDR} that 
Theorem \ref{theorem1.3} can be extended with suitable 
partial Dirichlet-to-Neumann maps.
We refer to \cite{IY2013} for the first uniqueness result in determining semilinear terms by partial Cauchy data on arbitrary subboundary.

Uniqueness results for recovering semilinear terms from full Cauchy data were obtained by Isakov and Sylvester \cite{IS} in three dimensions and by Isakov and Nachman \cite{IN} in two dimensions. These results  apply to nonlinearities of the form $a=a(x,u)$. 
For the sake of simplicity we only consider here the case 
$a=a(u)$. 
However we can expect that Theorem \ref{theorem1.3} can be extended to cover completely the uniqueness result in \cite{IS}, possibly under some additional conditions.

We point out that the uniqueness results for smooth semilinear terms using partial data in $\mathbb{R}^n$ ($n\geq 2$) were contained in the recent  papers by Krupchyk and Uhlmann \cite{KU}, and Lassas, Liimatainen, Lin and  Salo \cite{LLLS}.  These two references make use of higher order linearization procedure and contain a detailed overview of semilinear elliptic inverse problems together with a rich list of references. Without being exhaustive, we refer to \cite{IY2013, KN, LLLS1, MU, Su, SU} for other results concerning the unique determination of the nonlinear term in semilinear and quasilinear elliptic BVP's from boundary measurements. Similar inverse problem was studied in \cite{Is} for a semilinear parabolic equation and in \cite{CaK} for a quasilinear parabolic equation. Inverse problems for hyperbolic equations with various type of nonlinearities were considered in \cite{CLOP, HUZ, KLU, WZ}.

To our knowledge there are few stability results for the problem of determining the nonlinear term, appearing in partial differential equations, from boundary measurements. The determination of the nonlinear term in  a semilinear  parabolic equation, from the corresponding Dirichlet-to-Neumann map, was  studied by the first author and Kian \cite{CK}.  In \cite{CK} the authors establish a logarithmic stability estimate. A stability inequality of the determination of a nonlinear term in a parabolic equation from a single measurement was proved by the first and third authors and Ouhabaz in \cite{COY}.

The rest of this article is organized as follows. In Section \ref{sec:2} we give the proof of Theorem \ref{theorem1.1} and in Section \ref{sec:3} we prove Theorem \ref{theorem1.2}. Section \ref{sec:4} is devoted to establish a stability estimate for the linearized inverse problem. 
In Section \ref{sec:5}, we give the proof of Proposition \ref{proposition1.1} 
and Theorem \ref{theorem1.3} on the basis of Section \ref{sec:4}.

\section{Analysis of the semilinear BVP}\label{sec:2}

Prior to introducing the definition of variational solution of the BVP \eqref{2.1}, we prove the following lemma.

\begin{lemma}\label{lemma2.1}
Assume that $\alpha$ is arbitrary if $n=2$ and $\alpha\le (n+2)/(n-2)$ if $n\ge 3$. Let $a\in \mathscr{A}(\mathfrak{c},\alpha)$ and $\varphi\in L^{\alpha q^\ast}(\Omega )$, where $q^\ast=2n/(n+2)$ denotes the conjugate component of $q=2n/(n-2)$. Then the linear form on $H_0^1(\Omega )$ given by
\[
\ell (\phi)=\int_\Omega a(\varphi(x) )\phi(x)dx,\quad \phi \in H_0^1(\Omega ),
\]
is bounded with
\begin{equation}\label{2.2}
\|\ell\|_{H^{-1}(\Omega )}\le C(1+M^\alpha ),\quad \mbox{for any}\; \varphi \in B_{L^{\alpha q^\ast}}(M),
\end{equation}
where $C=C(\Omega, c_0,c_1,\alpha)>0$ is a constant.
\end{lemma}

\begin{proof}
Consider first the case $n\ge 3$. In that case $H_0^1(\Omega )$ is continuously embedded in $L^q(\Omega )$ with $q=2n/(n-2)$. We have in light of \eqref{a1} 
\[
\left|\int_\Omega a(\varphi(x) )\phi(x)dx\right|\le c_0\int_\Omega |\phi|dx+c_1\int_\Omega |\varphi|^\alpha |\phi|dx.
\]
Applying H\"older's inequality, we have
\[
\int_\Omega \vert \varphi\vert^{\alpha} \vert \vert \phi\vert dx 
\le \left( \int_{\Omega} \vert \varphi\vert^{\alpha q^\ast} dx \right)^{1/q^\ast}
\left( \int \vert\phi\vert^q dx \right)^{1/q}.
\]
Hence
\begin{align}
\left|\int_\Omega a(\varphi(x) )\phi(x)dx\right| 
&\le c_0\|\phi\|_{L^1(\Omega )}+c_1\|\varphi\|_{L^{\alpha q^\ast}(\Omega )}^\alpha \|\phi \|_{L^q(\Omega )}\label{2.3} 
\\
&\le c_\Omega\left(c_0+c_1\|\varphi\|_{L^{\alpha q^\ast}(\Omega )}^\alpha\right) 
\|\phi \|_{H_0^1(\Omega )} \nonumber
\\
&\le c_\Omega (c_0 + c_1M^{\alpha}) \|\phi \|_{H_0^1(\Omega )},   \nonumber
\end{align}
where we used that $H_0^1(\Omega )$ is continuously embedded  in $L^r(\Omega )$ for any $r\in[1,q]$. 
Taking the supremum over $\phi \in B_{H^1_0(\Omega)}(1)$ in both sides of \eqref{2.3} in order to obtain \eqref{2.2}.

The case $n=2$ can be carried out similarly by using that $H_0^1(\Omega )$ is continuously embedded in $L^r(\Omega )$ for any $r\ge 1$.
\end{proof}

Let $f\in H^{1/2}(\Gamma )$. We say that $u\in H^1(\Omega )$ is a variational solution of the BVP \eqref{2.1} if $u_{|\Gamma}=f$ (in the trace sense) and
\[
\int_\Omega \nabla u(x)\cdot\nabla \phi (x)dx+\int_\Omega a(u(x))\phi (x)dx=0,\quad \phi \in H_0^1(\Omega ).
\]

For $f\in H^{1/2}(\Omega )$, let $\mathscr{E}f\in H^1(\Omega )$ be its harmonic extension. That is, $v=\mathscr{E}f$ is the unique solution of the BVP
\[
\left\{
\begin{array}{ll}
-\Delta v=0\quad &\mbox{in}\; \Omega ,
\\
v=f &\mbox{on}\; \Gamma.
\end{array}
\right.
\]

Assume that we can find $w\in H_0^1(\Omega )$ satisfying
\begin{equation}\label{2.4}
\int_\Omega \nabla w(x) \cdot \nabla \phi(x)=-\int_\Omega a(w(x)+v(x))\phi(x)dx,\quad\mbox{for any}\; \phi \in H_0^1(\Omega ),
\end{equation}
An integration by parts yields 
\[
0=\int_\Omega \Delta v(x)\phi(x) dx=-\int_\Omega \nabla v(x)\cdot \nabla \phi (x)dx,\quad \mbox{for any}\; \phi \in C_0^\infty (\Omega ).
\]
Since $H_0^1(\Omega )$ is the closure of $C_0^\infty (\Omega )$ in $H^1(\Omega )$, we deduce that
\[
\int_\Omega \nabla v(x)\cdot \nabla \phi(x) dx=0,\quad \mbox{for any}\; \phi \in H_0^1(\Omega ).
\]
 We then obtain in light of \eqref{2.4}
\[
\int_\Omega \nabla (w(x)+v(x)) \cdot \nabla \phi(x)=-\int_\Omega a(w(x)+v(x))\phi(x)dx,\quad \mbox{for any}\; \phi \in H_0^1(\Omega ).
\]
In other words, $u=w+v$ is a variational solution of \eqref{2.1}.

\begin{theorem}\label{theorem2.1}
Assume that $\alpha$ is arbitrary if $n=2$ and $\alpha< (n+2)/(n-2)$ if $n\ge 3$. Let $a\in \mathscr{A}(\mathfrak{c},\alpha)$ and $f\in H^{1/2}(\Gamma )$. Then the BVP \eqref{2.1} has a unique variational solution $u_a(f)\in H^1(\Omega )$. Moreover, for any $M>0$, we have 
\begin{equation}\label{2.5}
\|u_a(f)\|_{H^1(\Omega )}\le C(1+M^\alpha),\quad \mbox{for any}\; f\in B_{H^{1/2}(\Gamma )}(M),
\end{equation}
where $C=C(\Omega,\mathfrak{c},\alpha)>0$ is a constant.
\end{theorem}

\begin{proof}
In light of the previous discussion, it is enough to prove that \eqref{2.4} has a solution $w\in H_0^1(\Omega )$ and \eqref{2.5} holds with $u_a(f)$ substituted by $w$.

Fix $w\in L^{\alpha q^\ast}(\Omega )$ and consider the variational problem: find $\psi \in H_0^1(\Omega )$ satisfying
\begin{equation}\label{LS1}
\int_\Omega \nabla \psi(x) \cdot \nabla \phi(x)=-\int_\Omega a(w(x)+v(x))\phi(x)dx,\quad\mbox{for any}\; \phi \in H_0^1(\Omega ).
\end{equation}
From Lemma \ref{lemma2.1} it follows that 
\[
\ell:\phi \mapsto \ell (\phi)=-\int_\Omega a(w(x)+v(x))\phi(x)dx
\]
defines a bounded linear form on $H_0^1(\Omega )$. Then Lax-Milgram's lemma,
which we apply to the functional on the left-hand side, 
guarantees that \eqref{LS1} has a unique solution $\psi \in H_0^1(\Omega )$.

Let $q=2n/(n-2)$ and $q^\ast=2n/(n+2)$ be its conjugate exponent to $q$ and define
\[
T:L^{\alpha q^\ast}(\Omega )\rightarrow L^{\alpha q^\ast}(\Omega ):w\mapsto Tw=\psi, 
\]
where $\psi\in H_0^1(\Omega )$ is the unique solution of the variational problem \eqref{LS1}. 

Assume that $H_0^1(\Omega)$ is endowed with the norm 
$\|\nabla h\|_{L^2(\Omega)}$. 
We obtain by taking $\phi=\psi$ in \eqref{LS1}
\[
\|\psi\|_{H_0^1(\Omega)}\le \|\ell \|_{H^{-1}(\Omega)}.
\]
This and inequality \eqref{2.2} in Lemma \ref{lemma2.1} yield
\[
\|Tw\|_{H_0^1(\Omega )}=\|\psi \|_{H_0^1(\Omega )}\le C(1+M^\alpha),\quad \mbox{for any}\; w\in B_{L^{\alpha q^\ast}(\Omega )}(M),
\]
where $C=C(\Omega,\mathfrak{c},\alpha)>0$ is a constant. 
That is, $T$ maps each bounded set of $L^{\alpha q^\ast}(\Omega )$ 
into a bounded set in $H_0^1(\Omega)$. 
Hence, according to Rellich-Kondrachov's theorem, $H_0^1(\Omega)$ is compactly 
embedded in $L^{\alpha q^\ast}(\Omega )$. Therefore, $T$ is a compact operator.

We are now going to show, with the help of Leray-Schauder's fixed point theorem, that $T$ has a fixed point. The crucial step consists in proving that the set
\[
K=\{ w\in L^{\alpha q^\ast}(\Omega );\; \mbox{there exists}\; \mu \in [0,1]\; \mbox{so that}\; w=\mu Tw\}
\]
is bounded in $L^{\alpha q^\ast}(\Omega )$.

Pick $w\in K$  and let $\mu \in [0,1]$ so that $w=\mu Tw$. According to the definition of $T$, $w$ ($\in H_0^1(\Omega)$) satisfies
\begin{equation}\label{2.6}
\int_\Omega |\nabla w(x)|^2dx=-\mu \int_\Omega a(w(x)+v(x)) w(x)dx.
\end{equation}
On the other hand, we have, for almost everywhere $x\in \Omega$,
\[
a(w(x)+v(x))=a(v(x))+\int_0^1a'(sw(x)+v(x))w(x)ds.
\]
This in \eqref{2.6} yields
\[
\int_\Omega |\nabla w(x)|^2dx=-\mu\int_\Omega a(v(x))w(x)dx-\mu \int_\Omega \left(\int_0^1a'(sw(x)+v(x))ds\right)w(x)^2dx.
\]
In light of assumption \eqref{a2} we obtain
\[
\int_\Omega |\nabla w(x)|^2dx\le -\mu\int_\Omega a(v(x))w(x)dx+c \int_\Omega w(x)^2dx
\]
which combined with Poincar\'e's inequality gives
\[
\int_\Omega |\nabla w(x)|^2dx\le -\mu\int_\Omega a(v(x))w(x)dx+c\lambda_1(\Omega)^{-1} \int_\Omega |\nabla w(x)|^2dx.
\]
Or equivalently
\[
(1-c\lambda_1(\Omega)^{-1} )\int_\Omega |\nabla w(x)|^2dx\le -\mu\int_\Omega a(v(x))w(x)dx.
\]
We then apply again Lemma \ref{lemma2.1} in order to obtain
\begin{align}
\|w\|_{L^{\alpha q^\ast}(\Omega )}\le C_0\|w\|_{H_0^1(\Omega)} \le C(1+&M^\alpha),\label{LS2}
\\
&\mbox{for any}\; w\in K\; \mbox{and}\; f\in B_{H^{1/2}(\Gamma )}(M),\nonumber
\end{align}
where $C_0=C_0(\Omega ,\alpha)>0$ and $C=C(\Omega,\mathfrak{c},\alpha)>0 $ are constants. 

In light of this inequality we can apply \cite[Theorem 11.3, page 280]{GT} to deduce that there exists $w^\ast\in H_0^1(\Omega )$ so that $w^\ast =Tw^\ast$. That is $w^\ast$ is the solution of the variational problem \eqref{2.1}. Furthermore, for any $f\in B_{H^{1/2}(\Gamma )}(M)$, we have from \eqref{LS2}
\[
\|w^\ast\|_{H_0^1(\Omega )}\le C(1+M^\alpha),
\]
where $C=C(\Omega,\mathfrak{c}, \alpha)>0$ is a constant.

We complete the proof by showing that \eqref{2.1} has at most one solution. To this end, let $u,\tilde{u}\in H_0^1(\Omega )$ be two solutions of \eqref{2.1} and set $v=u-\tilde{u}$. Taking into account that, for almost everywhere $x\in \Omega$, we have
\[
a(u(x))-a(\tilde{u}(x))=b(x) v(x) ,
\]
with
\[
b(x)=\int_0^1a'(x,\tilde{u}(x)+s(u(x)-\tilde{u}(x)))ds,
\]
we find that $v$ is the solution of the BVP
\[
\left\{
\begin{array}{ll}
-\Delta v+bv=0\quad &\mbox{in}\; \Omega ,
\\
v=0 &\mbox{on}\; \Gamma.
\end{array}
\right.
\]
Green's formula then yields
\[
\int_\Omega |\nabla v(x)|^2dx+\int_\Omega b(x)v(x)^2 dx=0.
\]
Hence
\[
\int_\Omega |\nabla v(x)|^2dx=-\int_\Omega b(x)v(x)^2dx\le c\int_\Omega v(x)^2dx\le c\lambda_1(\Omega )^{-1}\int_\Omega |\nabla v(x)|^2dx.
\]
By assumption $c\lambda_1(\Omega) ^{-1}<1$, we reach $v=0$.
\end{proof}

Theorem \ref{theorem1.1} will then follow from the following lemma.

\begin{lemma}\label{lemma2.2}
 Assume that $\alpha$ is arbitrary if $n=2$ and $\alpha\le n/(n-2)$ if $n\ge 3$. Let $a\in \mathscr{A}(\mathfrak{c},\alpha )$ and $f\in H^{3/2}(\Gamma )$. Then $u_a(f)\in H^2(\Omega )$ and
 \begin{equation}\label{2.7}
\|u_a(f)\|_{H^2(\Omega )}\le C(1+M+M^\alpha),\quad \mbox{for any}\; f\in B_{H^{3/2}(\Gamma )}(M),
\end{equation}
where $C=C(\Omega,\mathfrak{c},\alpha)>0$ is a constant. 
\end{lemma}

\begin{proof}
In this proof $C=C(\Omega,\mathfrak{c},\alpha,M)>0$ is a generic constant.

Consider the case $n\ge 3$. By \eqref{a1} we have, for almost everywhere $x\in \Omega$,
\[
[a\circ u_a(f)(x)]^2\le 2c_0^2+2c_1^2|u_a(f)(x)|^{2\alpha} .
\]
Using that $2\alpha \le 2n/(n-2)$ and $H^1(\Omega )$ is continuously embedded in $L^{2\alpha}(\Omega )$, we deduce that $a\circ u_a(f)\in L^2(\Omega )$ and from \eqref{2.5}, we obtain
\begin{equation}\label{2.8}
\|a\circ u_a(f)\|_{L^2(\Omega )}\le C(1+M^\alpha).
\end{equation}
From the elliptic regularity (e.g., \cite[Theorem 5.4, page 165]{LM}),
we deduce that $u_a(f)\in H^2(\Omega)$ and
\begin{equation}\label{2.9}
\|u_a(f)\|_{H^2(\Omega )}\le c_\Omega \left(\|f\|_{H^{3/2}(\Gamma )}+\|a\circ u_a(f)\|_{L^2(\Omega )}\right).
\end{equation}

Thus, inequalities \eqref{2.8} and \eqref{2.9} yield \eqref{2.7} 
in a straightforward manner.

The case $n=2$ can be treated similarly using that $H^1(\Omega )$ is 
continuously embedded in $L^r(\Omega )$ for any $r\ge 1$.
\end{proof}

\section{Dirichlet-to-Neumann map}\label{sec:3}

We first observe that by the help of Theorem \ref{theorem2.1} and Lemma \ref{lemma2.2} we can define the Dirichlet-to-Neumann map associated to $a\in \mathscr{A}(\mathfrak{c},\alpha)$. Precisely we have the following corollary.

\begin{corollary}\label{corollary2.1}
Assume that $\alpha$ is arbitrary if $n=2$ and $\alpha\le n/(n-2)$ if $n\ge 3$. For any $a\in \mathscr{A}(\mathfrak{c},\alpha)$ and $j=0,1$,  we can define the mapping
\[
\Lambda_a:H^{j+1/2}(\Gamma )\rightarrow H^{j-1/2}(\Gamma ):f\mapsto \partial_\nu u_a(f).
\]
Moreover, for any $M>0$,
\begin{equation}\label{2.9.1}
\|\Lambda_a(f)\|_{H^{j-1/2}(\Gamma )} \le C(1+M+M^\alpha),\quad \mbox{for any}\; f\in B_{H^{j+1/2}(\Gamma )}(M)
\end{equation}
where $C=C(\Omega,\mathfrak{c},\alpha)$ is a constant.
\end{corollary}

We recall that $C^{0,\theta}(\overline{\Omega})$, $0<\theta \le 1$, is the usual vector space of functions that are H\"older continuous on $\overline{\Omega}$ with exponent $\theta$. This space is usually endowed with its natural norm
\[
\|w\|_{C^{0,\theta}(\overline{\Omega})}=\|w\|_{C(\overline{\Omega})}+\sup_{x,y\in \overline{\Omega},\; x\ne y}\frac{|w(x)-w(y)|}{|x-y|^\theta}.
\]

Taking into account that $H^2(\Omega )$ is continuously embedded in $C^{0,1/2}(\overline{\Omega})$, for $n=2,3$, in view of Lemma \ref{lemma2.2} we obtain:

\begin{corollary}\label{corollary2.2}
 Assume that $\alpha$ is arbitrary if $n=2$ and $\alpha\le 3$ if $n=3$. Let $a\in \mathscr{A}(\mathfrak{c},\alpha)$, $M>0$ and $f\in B_{H^{3/2}(\Gamma )}(M)$. Then $u_a(f)\in C^{0,1/2}(\overline{\Omega})$ and
 \begin{equation}\label{2.10}
\|u_a(f)\|_{C^{0,1/2}(\overline{\Omega})}\le C(1+M+M^\alpha),
\end{equation}
where $C=C(\Omega,\mathfrak{c},\alpha)>0$ is a constant.
\end{corollary}

\begin{lemma}\label{lemma2.3}
(i) Assume that $n>4$, 
$n/2<p<n$ and $\alpha \le q/p$ with $q=2n/(n-4)$. 
Let $a\in \mathscr{A}(\mathfrak{c},\alpha)$, $M>0$ and $f\in B_{W^{2-1/p,p}(\Gamma )}(M)$. Then $u_a(f)\in W^{2,p}(\Omega )\cap C^{0,\beta}(\overline{\Omega})$, with $\beta =2-n/p$, and
 \begin{equation}\label{2.11}
\|u_a(f)\|_{W^{2,p}(\Omega )}+\|u_a(f)\|_{C^{0,\beta}(\overline{\Omega})}\le C(1+M+M^\alpha),
\end{equation}
where $C=C(\Omega, \mathfrak{c},\alpha, p)$ is a constant.
\\
(ii) Assume that $n=4$, $2<p<4$, $1\le r<2$, $q=2r/(2-r)$ and $\alpha \le q/p$. Let $a\in \mathscr{A}(\mathfrak{c},\alpha)$, $M>0$ and $f\in B_{W^{2-1/p,p}(\Gamma )}(M)$. Then $u_a(f)\in W^{2,p}(\Omega )\cap C^{0,\beta}(\overline{\Omega})$, with $\beta =2-4/p$, and
 \begin{equation}\label{2.12}
\|u_a(f)\|_{W^{2,p}(\Omega )}+\|u_a(f)\|_{C^{0,\beta}(\overline{\Omega})}\le C(1+M+M^\alpha),
\end{equation}
where $C=C(\Omega,\mathfrak{c},\alpha,p,r)>0$ is a constant.
\end{lemma}

\begin{proof}
(i) In this part $C=C(\Omega, \mathfrak{c},\alpha,p)>0$ is a generic constant.

Noting that $q/p<n/(n-2)$, we obtain from Lemma \ref{lemma2.2} that $u_a(f)\in H^2(\Omega )$ and, since $H^2(\Omega )$ is continuously embedded in $L^q(\Omega )$ with $q=2n/(n-4)$, $u_a(f)\in L^q(\Omega )$. Consequently, using \eqref{a1}, \eqref{2.7} and the assumption on $\alpha$, we obtain $a\circ u_a(f)\in L^p(\Omega )$ and
\begin{equation}\label{2.13}
\|a\circ u_a(f)\|_{L^p(\Omega )}\le C(1+M+M^\alpha).
\end{equation}
We obtain by applying \cite[Theorem 9.15, page 241]{GT} that $u_a(f)\in W^{2,p}(\Omega)$ and, since $W^{2,p}(\Omega)$ is continuously embedded in $C^{0,\beta}(\overline{\Omega})$, we conclude that $u_a(f)\in C^{0,\beta}(\overline{\Omega})$.

A combination of \cite[(9.46), page 242]{GT} and \eqref{2.13} yields in straightforward manner 
\[
\|u_a(f)\|_{W^{2,p}(\Omega)}\le C(1+M+M^\alpha).
\]
Hence \eqref{2.11} follows.

(ii) Let $n=4$ and $1\le r<2$. As $q/p<2$, we obtain from Lemma \ref{lemma2.2} that $u_a(f)\in H^2(\Omega )$. Since $H^2(\Omega )$ is continuously embedded in $W^{2,r}(\Omega )$ and $W^{2,r}(\Omega )$ is continuously embedded in $L^q(\Omega )$ with $q=2r/(2-r)$, we conclude that $H^2(\Omega )$ is continuously embedded in $L^q(\Omega )$. Hence, if $\alpha p\le q$, for some $2<p<4$, then $u\circ u_a(f)$ is in $L^p(\Omega )$. The rest of the proof is quite similar to that of (i).
\end{proof}

We end this section by noting that Theorem \ref{theorem1.2} follows readily from Corollary \ref{corollary2.2} and Lemma \ref{lemma2.3}.

\section{Linearized inverse problem}\label{sec:4}

Some parts of this section are borrowed from \cite{Chou}. The main novelty of the results in this section consists in constructing  complex geometric optic solutions in $W^{2,r}(\Omega )$ for any $r\in [2,\infty )$.

All functions we consider in this section are assumed to be complex-valued.

Fix $\xi \in \mathbb{S}^n$, $\mathfrak{q}\in L^\infty(\Omega )$ and, for $h>0$, consider the operator
\[
P_h=P_h(\mathfrak{q},\xi )=e^{x\cdot \xi /h}h^2(-\Delta +\mathfrak{q})e^{-x\cdot \xi/h}.
\]
Clearly we can write $P_h$ in the form
\[
P_h=-h^2\Delta +2h\xi \cdot \nabla -1+h^2\mathfrak{q}.
\]

\begin{lemma}\label{lemma-bp1}
(Carleman inequality) Let $M>0$. Then there exists a constant $c_\Omega>0$ so that, for any $q\in B_{L^\infty (\Omega )}(M)$, $0<h<h_0=c_\Omega/(2M)$ and $u\in C_0^\infty (\Omega )$, we have
\begin{equation}\label{bp1}
h\|u\|_{L^2(\Omega )}\le 2c_\Omega ^{-1}\|P_hu\|_{L^2(\Omega )}.
\end{equation}
\end{lemma}

\begin{proof}
Let $P_h^0=P_h(0,\xi )$. For $u\in C_0^\infty (\Omega )$, we have
\begin{align}
 \|P_h^0u\|_{L^2(\Omega )}^2&=\|(h^2\Delta +1)u\|_{L^2(\Omega)}^2\label{bp2}
 \\
 &-4h\Re ((h^2\Delta +1)u,\xi\cdot \nabla u)_{L^2(\Omega)}+h^2\|\xi \cdot \nabla u\|_{L^2(\Omega)}^2.\nonumber
 \end{align}
Simple integrations by parts yields
\[
\Re ((h^2\Delta +1)u,\xi\cdot \nabla u)_{L^2(\Omega)}=0.
\]

This in \eqref{bp2} gives
\begin{equation}\label{bp3}
 \|P_h^0u\|_{L^2(\Omega )}^2\ge h^2 \|\xi \cdot \nabla u\|_{L^2(\Omega)}^2.
 \end{equation}

From Poincar\'e's inequality and its proof, we have
\[
\|\xi \cdot \nabla u\|_{L^2(\Omega)}^2\ge c_\Omega \|u\|_{L^2(\Omega )}.
\]
This and \eqref{bp3} imply 
\begin{equation}\label{bp4}
\|P_h^0u\|_{L^2(\Omega )}\ge c_\Omega h\|u\|_{L^2(\Omega )}.
\end{equation}

Pick $\mathfrak{q}\in B_{L^\infty (\Omega )}(M)$. Since
\[
\|P_h^0\|_{L^2(\Omega )}\le \|P_hu\|_{L^2(\Omega)} +h^2M\|u\|_{L^2(\Omega)},
\]
we obtain from \eqref{bp4}
\[
c_\Omega h\|u\|_{L^2(\Omega )}\le \|P_hu\|_{L^2(\Omega )}+h^2M\|u\|_{L^2(\Omega)}.
\]
This inequality yields \eqref{bp1} in a straightforward manner.
\end{proof}

\begin{proposition}\label{proposition-bp1}
Let $M>0$. There exists a constant $c_\Omega>0$ so that, for any $\mathfrak{q}\in B_{L^\infty (\Omega )}(M)$  and $0<h<h_0=c_\Omega/(2M)$, we find $w\in L^2(\Omega )$ satisfying 
\[
\left[e^{x\cdot \xi /h}(-\Delta +\mathfrak{q})e^{-x\cdot \xi/h}\right]w=f
\]
and
\begin{equation}\label{bp6}
\|w\|_{L^2(\Omega )}\le 2c_\Omega^{-1}h\|f\|_{L^2(\Omega )}.
\end{equation}
\end{proposition}

\begin{proof}
Pick $\mathfrak{q}\in B_{L^\infty (\Omega )}(M)$ and $\xi \in \mathbb{S}^{n-1}$. Let $H=P_h^\ast (C_0^\infty (\Omega ))$ that we consider as a subspace of $L^2(\Omega )$. We observe that if $P_h=P_h (\mathfrak{q},\xi)$ then $P_h^\ast=P_h(\overline{\mathfrak{q}},-\xi)$. Therefore inequality \eqref{bp1} holds when $P_h$ is substituted by $P_h^\ast$.

Let $f\in L^2(\Omega )$ and define on $H$ the linear form
\[
\ell (P_h^\ast v)=(v,h^2f)_{L^2(\Omega)},\quad v\in C_0^\infty (\Omega ).
\]

From Lemma \ref{lemma-bp1}, $\ell$ is  bounded with 
\[
|\ell (P_h^\ast v)|\le h^2\|f\|_{L^2(\Omega )}\|v\|_{L^2(\Omega )}\le 2c_\Omega^{-1}h\|f\|_{L^2(\Omega )}\|P_h^\ast v\|_{L^2(\Omega )}.
\]
Hence, according to the Hahn-Banach extension theorem, there exists 
a linear form $L$ extending $\ell$ to $L^2(\Omega )$ so that $\|L\|_{\left[L^2(\Omega )\right]'}=\|\ell\|_{H}$. In consequence
\begin{equation}\label{bp4}
\|L\|_{\left[L^2(\Omega )\right]'}\le 2c_\Omega^{-1}h\|f\|_{L^2(\Omega )}.
\end{equation}

Applying Riesz's representation theorem, we find $w\in L^2(\Omega )$ 
such that 
\begin{equation}\label{bp5}
\|w\|_{L^2(\Omega )}=\|L\|_{\left[L^2(\Omega )\right]'}
\end{equation}
and
\[
(P_h^\ast v,w)_{L^2(\Omega )}=L(P_h^\ast v)=\ell (P_h^\ast v)=(v,h^2f)_{L^2(\Omega)},\quad v\in C_0^\infty (\Omega ).
\]
Hence 
\[
\left[e^{x\cdot \xi /h}(-\Delta +\mathfrak{q})e^{-x\cdot \xi/h}\right]w=f. 
\]
We complete the proof by noting that \eqref{bp6} is obtained by combining \eqref{bp4} and \eqref{bp5}. 
\end{proof}

\begin{proposition}\label{proposition4.2}
Let $\mathcal{O}\Supset \Omega$, $M>0$, $\mathfrak{q}\in B_{L^\infty (\Omega)}(M)$ and $u\in L^2(\mathcal{O})$ satisfying 
\[
(-\Delta +\mathfrak{q}\chi_\Omega)u=0\; \mbox{in}\; \mathscr{D}'(\mathcal{O}). 
\]
(i) We have $u\in H_{\rm loc}^1(\mathcal{O})$ and, for any $\Omega \Subset \Omega_1 \Subset \Omega_2\Subset \mathcal{O}$, we have the following interior Caccioppoli type inequality
\[
\|u\|_{H^1(\Omega_1 )}\le C(1+M)\|u\|_{L^2(\Omega _2)},
\]
where $C=C(\Omega ,\mathcal{O}, d)>0$ is a constant with $d=\mbox{dist}(\overline{\Omega_1}, \partial \Omega_2)$.
\\
(ii) We have $u\in W_{\rm loc}^{2,r}(\mathcal{O})$ for any $1<r<\infty$,
 \[
\|u\|_{W^{2,r}(\Omega )}\le C(1+M)^2\|u\|_{L^2(\mathcal{O})},
\]
where $C=C(\Omega , \mathcal{O},r)>0$ is a constant.
\end{proposition}

\begin{proof}
Fix $\phi \in C_0^\infty(\mathcal{O})$. Then $v=\phi u$ is the solution of the BVP
\[
\left\{
\begin{array}{ll}
-\Delta v=-\mathfrak{q}\chi_\Omega \phi u-2\nabla u\cdot \nabla \phi -\Delta \phi u \quad &\mbox{in}\; \mathcal{O},
\\
v=0 &\mbox{on}\; \partial \mathcal{O}.
\end{array}
\right.
\]
Since 
\[
-\mathfrak{q}\chi_{\Omega''}u-2\nabla u\cdot \nabla \phi -\Delta \phi u\in H^{-1}(\mathcal{O}),
\]
we obtain $\phi u\in H_0^1(\mathcal{O})$.

Next, pick $\psi \in C_0^\infty (\Omega _2)$ satisfying $0\le \psi \le 1$, $\psi =1$ in a neighborhood of $\overline{\Omega_1}$ and $|\nabla \psi|\le \kappa$, where  $\kappa >0$ is a constant only depending on $\mbox{dist}(\overline{\Omega_1},\partial \Omega _2)$. Let $(v_k)$ be a sequence in $C_0^\infty (\Omega _2)$ converging to $\psi^2u$ in $H^1(\Omega _2)$. We pass to the limit in the identity 
\[
\int_{\Omega _2}\nabla u\cdot \nabla \overline{v_k}dx+\int_{\Omega _2}qu\overline{v_k}dx=0
\]
in order to obtain
\[
\int_{\Omega _2}\nabla u\cdot \nabla (\psi ^2\overline{u}) dx+\int_{\Omega _2}\mathfrak{q}\psi^2|u|^2dx=0.
\]
Hence
\begin{equation}\label{Ca1}
\int_{\Omega _2} |\psi \nabla u|^2dx=-2\int_{\Omega _2}\psi \nabla u\cdot \overline{u}\nabla \psi -\int_{\Omega _2}q\psi^2|u|^2dx.
\end{equation}
For any $\epsilon >0$, we have
\[
|\psi \nabla u\cdot \overline{u}\nabla \psi |\le (\epsilon /2)|\psi \nabla u| ^2+ (1/(2\epsilon ))|u|^2|\nabla \psi|^2.
\]
The particular choice $\epsilon =1/2$ yields
\[
|\psi \nabla u\cdot \overline{u}\nabla \psi |\le (1/4)|\psi \nabla u| ^2+ |u|^2|\nabla \psi|^2.
\]
This inequality together with \eqref{Ca1} give
\[
\int_{\Omega_1}  |\psi \nabla u|^2dx\le \int_{\Omega _2} |\psi \nabla u|^2dx\le 2(M+\kappa^2) \int_{\Omega _2} |u|^2dx.
\]
(ii) Let $\Omega \Subset \Omega _1\Subset \Omega '\Subset \Omega _2\Subset \mathcal{O}$ be subdomains.   Let $\psi \in C_0^\infty (\Omega ')$ satisfying $0\le \psi \le 1$, $\psi =1$ in a neighborhood of $\overline{\Omega_1}$. Then $\psi u$ is the solution of the BVP
\[
\left\{
\begin{array}{ll}
-\Delta (\psi u)=-\mathfrak{q}\chi_\Omega u-2\nabla u\cdot \nabla \psi -\Delta \psi u \quad &\mbox{in}\; \Omega ',
\\
u=0 &\mbox{on}\; \partial \Omega '.
\end{array}
\right.
\]
From $H^2$ interior estimates (see for instance \cite[Section 8.5]{RR}) $u=\psi u\in H^2(\Omega_1 )$ and there exists a constant $c_\Omega >0$ so that
\[
\|u\|_{H^2(\Omega_1 )}\le c_\Omega\|q\chi_\Omega u+2\nabla u\cdot \nabla \psi +\Delta \psi u\|_{L^2(\Omega ')}.
\]
Hence
\[
\|u\|_{H^2(\Omega_1 )}\le C(1+M)\|u\|_{H^1(\Omega ')},
\]
where $C=C(\Omega,\mathcal{O} , \Omega_1,\Omega ')>0$ is a constant.

This inequality  combined with (i) yields
\[
\|u\|_{H^2(\Omega_1 )}\le C(1+M)\|u\|_{L^2(\Omega _2)},
\]
where $C=C(\Omega,\mathcal{O}, \Omega_1,\Omega ',\Omega_2)>0$ is a constant.

Assume $n>2$ and set $r_0=(2n)/(n-2)$. As $H^1(\Omega ')$ is continuously embedded in $L^r(\Omega ')$ for $r\in [1,r_0]$, we have 
\[
-\mathfrak{q}\chi_{\Omega}u-2\nabla u\cdot \nabla \psi -\Delta \psi u\in L^r(\Omega '),
\]
We then obtain by applying \cite[Theorem 9.15, page 241]{GT} that $u\in W^{2,r}(\Omega )$. Furthermore, \cite[Lemma 9.17, page 242]{GT} gives
\begin{align*}
\|u\|_{W^{2,r}(\Omega _1)}&\le C\|\mathfrak{q}\chi_{\Omega}u+2\nabla u\cdot \nabla \psi +\Delta \psi u\|_{L^r(\Omega ')}
\\
&\le C(1+M)\|u\|_{H^2(\Omega ')}
\\
&\le C(1+M)^2\|u\|_{L^2(\Omega _2)},
\end{align*}
where $C=C(\Omega,\mathcal{O}, \Omega_1,\Omega ',\Omega_2)>0$ is a constant.

If $r_0<n$, we set $r_1=(nr_0)/(n-r_0)$ and we repeat the preceding step where $r_0$ is substituted by $r_1$. We obtain that $u\in W^{2,r}(\Omega _1)$ for $r\in [1,r_1]$ and
\[
\|u\|_{W^{2,r}(\Omega _1)}\le C(1+M)^2\|u\|_{L^2(\Omega _2)}.
\]
If $r_0<n$ and $r_1<n$, $r_2$ given by $r_2=(nr_1)/(n-r_1)$ satisfies
\[
r_2=r_1+\frac{r_1^2}{n-r_1}\ge r_0+2\frac{r_0^2}{n-r_0},
\]
where we used that the mapping $t\in [0,n[\mapsto t^2/(n-t)$ is increasing. By induction in $k\ge 1$, if $r_j<n$ for $0\le j\le k$ we set $r_{k+1}=(nr_k)/(n-r_k)$. In that case we have
\[
r_{k+1}\ge r_0+(k+1)\frac{r_0^2}{n-r_0}.
\]
Since the right hand side of this inequality tends to $\infty$ when $k$ goes to $\infty$, we find a non negative integer $k_n$ so that $r_j<n$ if $0\le j\le k_n-1$ and
$r_{k_n}\ge n$.

We repeat the preceding arguments from $r_0$ until $r_{k_n-1}$. We obtain $u\in W^{2,r}(\Omega_1)$ with $r\in [1,r_{k_n-1}]$. If $r_{k_n}>n$, we complete the proof since $W^{1,r_{k_n}}(\Omega )$ is continuously embedded in  $L^\infty (\Omega )$. Otherwise $r_{k_n+1}>n$ and we end up getting the expected result by a last step.
\end{proof}

\begin{theorem}\label{theorem-bp1}
Let $M>0$ and $1<r<\infty$. Then there exist $C=C(\Omega ,r)$, $c_\Omega >0$, $\kappa=\kappa (\Omega )$ so that, for any $\mathfrak{q}\in B_{L^\infty (\Omega )}(M)$, $\xi,\zeta \in \mathbb{S}^{n-1}$ satisfying $\xi\bot\zeta$  and $0<h\le h_0=c_\Omega/(2M)$, the equation
\[
(-\Delta +\mathfrak{q})u=0\quad \mbox{in}\; \Omega 
\]
admits a solution $u\in W^{2,r}(\Omega )$ of the form
\[
u=e^{-x\cdot\left(\xi +i\zeta\right)/h}(1+v),
\]
where $v\in W^{2,r}(\Omega )$ satisfies
\[
\|v\|_{L^2(\Omega )}\le 2c_\Omega ^{-1}h.
\]
Moreover, we have
\[
\|u\|_{W^{2,r}(\Omega )}\le C(1+M)^2e^{\kappa/h}.
\]
\end{theorem}

\begin{proof}
Fix $\mathcal{O}\Supset \Omega$ arbitrary. We first consider the equation
\begin{equation}\label{cgo1}
(-\Delta +\mathfrak{q}\chi_\Omega)u=0\quad \mbox{in}\; \mathcal{O}.
\end{equation}
If $u=e^{-x\cdot\left(\xi +i\zeta\right)/h}(1+v)$ then $v$ should verify
\begin{align*}
&\left[e^{x\cdot \xi /h}(-\Delta +\mathfrak{q}\chi_\Omega)e^{-x\cdot \xi/h}\right]\left(e^{-ix\cdot \zeta /h}v\right)
\\
&\qquad =-\left[e^{x\cdot \xi /h}(-\Delta +\mathfrak{q})e^{-x\cdot \xi/h}\right]\left(e^{-ix\cdot \zeta /h}\right)=-\mathfrak{q}\chi_\Omega e^{-ix\cdot \zeta/h}.
\end{align*}
By Proposition \ref{proposition-bp1}, with $\Omega$ and $\mathfrak{q}$ substituted respectively by $\mathcal{O}$ and $\mathfrak{\mathfrak{q}}\chi_\Omega$, we find $w\in L^2(\mathcal{O})$ so that  
\[
\left[e^{x\cdot \xi/h }(-\Delta +\mathfrak{q}\chi_\Omega)e^{-x\cdot \xi/h}\right]w=-\mathfrak{q}\chi_\Omega e^{-ix\cdot \zeta/h}
\]
and
\[
\|w\|_{L^2\mathcal{O})}\le 2c_\Omega^{-1}h.
\]
Let $v=e^{ix\cdot \zeta /h}w$. Then
\[
\|v\|_{L^2(\mathcal{O})}\le 2c_\Omega^{-1}h
\]
and $u=e^{-x\cdot\left(\xi +i\zeta\right)/h}(1+v)$ is a solution of \eqref{cgo1}. Furthermore, we apply Proposition \ref{proposition4.2} in order to obtain
\begin{align*}
\|u\|_{W^{2,r}(\Omega )}&\le C(1+M)\|e^{-x\cdot\left(\xi +i\zeta\right)/h}(1+v)\|_{L^2(\mathcal{O})}
\\
&\le C(1+M)e^{\kappa/h}.
\end{align*}
This completes the proof.
\end{proof}

When $\mathfrak{q}\in L^\infty (\Omega)$ satisfies $\mathfrak{q}\ge -c$ almost everywhere, we can easily verify, with the help of Poincar\'e's inequality, that $0$ does not belong to the spectrum of $-\Delta +\mathfrak{q}$ under Dirichlet boundary condition.  
For notational convenience we set
\[
\mathcal{Q}_c=\{\mathfrak{q}\in L^\infty(\Omega ) ;\; \mathfrak{q}\ge -c\;  \mbox{almost everywhere}\}.
\]

\begin{theorem}\label{theorem-existence}
Let $M>0$ and $2\le r<\infty$. For any $\mathfrak{q}\in \mathcal{Q}_c\cap B_{L^\infty(\Omega )}(M)$ and $f\in W^{2-1/r,r}(\Gamma)$, the BVP
\begin{equation}\label{ex1}
\left\{
\begin{array}{ll}
(-\Delta +\mathfrak{q})u=0\quad &\mbox{in}\; \Omega ,
\\
u=f &\mbox{on}\; \Gamma.
\end{array}
\right.
\end{equation}
admits a unique solution $u_\mathfrak{q}(f)\in W^{2,r}(\Omega )$. Furthermore
\begin{equation}\label{ex0}
\|u_\mathfrak{q}(f)\|_{W^{2,r}(\Omega )}\le C(1+M)\|f\|_{W^{2-1/r,r}(\Gamma)},
\end{equation}
where $C=C(\Omega,c ,r)>0$ is a constant.
\end{theorem}

\begin{proof}[Sketch of the proof]
Let $2\le r\le \infty$, $f\in W^{2-1/r,r}(\Gamma)$ and pick $F\in W^{2,r}(\Omega )$ so that $F=f$ on $\Gamma$ and $\|F\|_{W^{2,r}(\Omega )}\le 2 \|f\|_{W^{2-1/r,r}(\Gamma)}$. If $u$ is a solution of \eqref{ex1} then $v=u-F$ must be a solution of the BVP
\begin{equation}\label{ex2}
\left\{
\begin{array}{ll}
(-\Delta +\mathfrak{q})v=g:=\Delta F-\mathfrak{q}F\quad &\mbox{in}\; \Omega ,
\\
v=0 &\mbox{on}\; \Gamma.
\end{array}
\right.
\end{equation}
According to \cite[Sections 8.5 and 8.6]{RR}, the BVP \eqref{ex2} has a unique solution $v\in H^2(\Omega )$ so that
\begin{align}
\|v\|_{H^2(\Omega )}&\le c_\Omega \left(\|\mathfrak{q}u\|_{L^2(\Omega )}+\|g\|_{L^2(\Omega )}\right) \label{ex3}
\\
&\le C\left(\|\mathfrak{q}u\|_{L^2(\Omega )}+\|g\|_{L^r(\Omega )}\right)\nonumber
\\
&\le C(1+M)\left(\|u\|_{L^2(\Omega )}+ \|f\|_{W^{2-1/r,r}(\Gamma)}\right), \nonumber
\end{align}
where $C=C(\Omega,c,r)>0$ is a  constant.

On the other hand from \eqref{ex2} we obtain
\begin{align*}
\int_\Omega |\nabla u|^2dx&=-\int_\Omega \mathfrak{q}|u|^2dx+\int_\Omega g\overline{v}dx
\\
&\le c\int_\Omega |u|^2dx+\|g\|_{L^2(\Omega )}\|u\|_{L^2(\Omega )}.
\end{align*}
From Poincar\'e's inequality
\[
\lambda_1(\Omega )\int_\Omega |u|^2dx\le \int_\Omega |\nabla u|^2dx.
\]
Hence
\[
\|u\|_{L^2(\Omega )}\le (\lambda_1(\Omega)-c)^{-1}\|g\|_{L^2(\Omega )}.
\]
This in \eqref{ex3} gives
\[
\|v\|_{H^2(\Omega )}\le C(1+M)\|f\|_{W^{2-1/r,r}(\Gamma )}.
\]
Here and henceforward $C=C(\Omega ,c,r)>0$ is a generic constant.

As in Proposition \ref{proposition4.2} we discuss separately cases $n=2,3$, $n=4$ and $n>4$. If $n>4$, we know that $H^2(\Omega )$ is continuously embedded in $L^s(\Omega )$ for $s\in [2, (2n)/(n-4)]$. We then apply \cite[Theorem 9.15 page 241 and Theorem 9.17 page 242]{RR}. We conclude  that $v\in W^{2,s}(\Omega )$ with
\begin{align*}
\|v\|_{W^{2,s}(\Omega )}&\le C\|-\mathfrak{q}v+g\|_{L^s(\Omega)}
\\
&\le C(\|v\|_{L^s(\Omega )}+\|g\|_{L^r(\Omega )})
\\
&\le C(\|v\|_{H^2(\Omega )}+\|g\|_{L^r(\Omega )})
\\
&\le C(1+M)\|f\|_{W^{2-1/r,r}(\Gamma)}.
\end{align*}
The rest of the proof is quite similar to that Proposition \ref{proposition4.2}. That is based on the iterated $W^{2,s}$ regularity and the corresponding a priori estimate. Finally, once we proved 
\[
\|v\|_{W^{2,r}(\Omega )}\le C(1+M)\|f\|_{W^{2-1/r,r}(\Gamma)},
\]
we end up getting the expected inequality by noting that
\[
\|u\|_{W^{2,r}(\Omega)}\le  \|v\|_{W^{2,r}(\Omega )}+\|F\|_{W^{2,r}(\Omega )}.
\]
The proof in then complete.
\end{proof}

In light of Theorem \ref{theorem-existence}, we can define 
the Dirichlet-to-Neumann map associated to $r\in [2,\infty )$ and 
$\mathfrak{q}\in \mathcal{Q}_c$ as follows
\[
\Lambda _\mathfrak{q}^r:f\in W^{2-1/r,r}(\Gamma)\mapsto \partial_\nu u_\mathfrak{q}(f)\in W^{1-1/r,r}(\Gamma ).
\]
Additionally, estimate \eqref{ex0} yields
\[
\|\Lambda_\mathfrak{q}^r \|\le C(1+M),\quad \mbox{for any}\; \mathfrak{q}\in \mathcal{Q}_c\cap B_{L^\infty(\Omega )}(M),
\]
where $C=C(\Omega ,c,r)>0$ is a constant and $\|\Lambda_q^r \|$ denotes the norm of $\Lambda_\mathfrak{q}^r$ in $\mathscr{B}(W^{2-1/r,r}(\Gamma),W^{1-1/r,r}(\Gamma ))$.

We also define, for $\mathfrak{q}\in \mathcal{Q}_c$ and $r\in [2,\infty )$,
\[
\mathscr{S}_\mathfrak{q}^r=\{u\in W^{2,r}(\Omega);\; (-\Delta +\mathfrak{q})u=0\; \mbox{in}\; \Omega\}.
\]

\begin{lemma}\label{integral-identity}
(Integral identity) For $r\in [2,\infty )$, $\mathfrak{q},\tilde{\mathfrak{q}}\in \mathcal{Q}_c$, $u\in \mathscr{S}_\mathfrak{q}^r$ and $\tilde{u}\in \mathscr{S}_{\tilde{\mathfrak{q}}}^r$, we have
\begin{equation}\label{ii}
\int_\Omega ({\tilde{\mathfrak{q}}}-\mathfrak{q})u\tilde{u}dx =\int_\Gamma (\Lambda_{\tilde{\mathfrak{q}}}^r-\Lambda_\mathfrak{q}^r)(u_{|\Gamma})\tilde{u}d\sigma(x) .
\end{equation}
\end{lemma}

\begin{proof}
Let $v=u_{\tilde{\mathfrak{q}}}(u_{|\Gamma})$. We obtain by applying Green's formula
\begin{equation}\label{ii1}
\int_\Gamma \partial_\nu (u-v)\tilde{u} d\sigma (x)= \int_\Omega (\mathfrak{q}u-\tilde{\mathfrak{q}}v)\tilde{u}dx+\int_\Omega \nabla (u-v)\cdot \nabla \tilde{u}dx 
\end{equation}
and
\begin{equation}\label{ii2}
0=\int_\Gamma \partial_\nu {\tilde{u}}(u-v)d\sigma(x) = \int_\Omega \tilde{\mathfrak{q}} \tilde{u}(u-v)dx+\int_\Omega \nabla \tilde{u}\cdot \nabla (u-v)dx.
\end{equation}

Identity \eqref{ii2} yields
\[
\int_\Omega \nabla \tilde{u}\cdot \nabla (u-v)dx=-\int_\Omega \tilde{\mathfrak{q}} \tilde{u}(u-v)dx.
\]
This inequality in \eqref{ii1} gives
\[
\int_\Gamma \partial_\nu (u-v)\tilde{u} d\sigma (x) = \int_\Omega (\mathfrak{q}-{\tilde{\mathfrak{q}}})u\tilde{u}dx.
\]
We end up getting the expected identity because 
\[ 
\int_\Gamma \partial_\nu (u-v)\tilde{u} d\sigma (x)=\int_\Gamma (\Lambda_{\tilde{\mathfrak{q}}}^r-\Lambda_\mathfrak{q}^r)(u_{|\Gamma})\tilde{u}d\sigma(x).
\]
\end{proof}

The following observation will be useful in the sequel: if $w\in H^t(\Omega )$, $0<t<1/2$, then $w\chi_\Omega \in H^t(\mathbb{R}^n)$ (see \cite[page 31]{Gr}).

\begin{theorem}\label{theorem-stab}
Let $M>0$, $r\in [2,\infty )$ and  $0<s<1/2$ and assume that $n\ge 3$. Then there exist two constants $C=C(\Omega,r,s)>0$ and $\rho_0=\rho_0(\Omega ,M)$ so that, for any $\mathfrak{q},\tilde{\mathfrak{q}}\in B_{H^s(\Omega)\cap L^\infty (\Omega)}(M)\cap \mathcal{Q}_c$, we have
\[
C\|\mathfrak{q}-\tilde{\mathfrak{q}}\|_{L^2(\Omega )} \le 1/\rho^\gamma +\mathfrak{D}(1+M)^4 e^{\kappa\rho} ,\quad  \rho \ge \rho_0.
\]
with $\gamma =\min (1/2 ,\sigma/n)$ and 
\[
\mathfrak{D} =\|\Lambda_\mathfrak{q}^r-\Lambda_{\tilde{\mathfrak{q}}}^r\|_{\mathscr{B}\left(W^{2-1/r,r}(\Gamma ),W^{1-1/r,r}(\Gamma )\right)}.
\]
\end{theorem}

\begin{proof}
Pick $\mathfrak{q},\tilde{\mathfrak{q}}\in B_{H^s(\Omega)\cap L^\infty (\Omega)}(M)\cap \mathcal{Q}_c$. Let $k,\tilde{k}\in \mathbb{R}^n\setminus\{0\}$ and $\xi \in \mathbb{S}^{n-1}$ so that $k\bot \tilde{k}$, $k\bot \xi$ and $\tilde{k}\bot \xi$ (this is possible because $n\ge 3$). We assume that $|\tilde{k}|=\rho$ with $\rho \ge \rho_0=h_0^{-1}$ where $h_0$ is as Theorem \ref{theorem-bp1}. Let then
\[
h=h(\rho)=\frac{1}{(|k|^2/4+\rho^2)^{1/2}}\; (\le h_0).
\]
Set
\[
\zeta =h(k/2+\tilde{k}),\quad \tilde{\zeta}=h(k/2-\tilde{k})
\]
As we have seen in the proof of Theorem \ref{theorem-bp1}, $\zeta,\tilde{\zeta}\in \mathbb{S}^{n-1}$, $\zeta \bot \xi$, $\tilde{\zeta} \bot \xi$ and $\zeta+\tilde{\zeta}=hk$.

By Theorem \ref{theorem-bp1}, the equation
\[
(-\Delta +\mathfrak{q})u=0\quad \mbox{in}\; \Omega 
\]
admits a solution $u\in W^{2,r} (\Omega )$ of the form
\[
u=e^{-x\cdot(\xi +i\zeta )/h}(1+v)
\]
so that, for some constants $C=C(\Omega ,r)>0$ and $\kappa =\kappa (\Omega)$,
\begin{equation}\label{u9}
\|v\|_{L^2(\Omega )}\le Ch
\end{equation}
and
\begin{equation}\label{u9.1}
\|u\|_{W^{2,r}(\Omega )}\le C(1+M)^2e^{\kappa/h}.
\end{equation}

Similarly, the equation
\[
(-\Delta +\tilde{\mathfrak{q}})u=0\quad \mbox{in}\; \Omega 
\]
admits a solution $\tilde{u}\in W^{2,r} (\Omega )$ of the form
\[
\tilde{u}=e^{-x\cdot(-\xi +i\tilde{\zeta} )/h}(1+\tilde{v}),
\]
with
\begin{equation}\label{u10}
\|\tilde{v}\|_{L^2(\Omega )}\le Ch
\end{equation}
and
\begin{equation}\label{u10.1}
\|\tilde{u}\|_{W^{2,r}(\Omega )}\le C(1+M)^2e^{\kappa/h}.
\end{equation}

We use the following temporary notations 
\[
w=(v+\tilde{v}+v\tilde{v})e^{-ix\cdot k},\quad g=u_{|\Gamma},\quad \tilde{g}=\tilde{u}_{|\Gamma} .
\]

We find by applying the integral identity \eqref{ii}
\[
\int_\Omega (\mathfrak{q}-\tilde{\mathfrak{q}}) e^{-ix\cdot k}dx= -\int_\Omega (\mathfrak{q}-\tilde{\mathfrak{q}})wdx+\int_\Gamma (\Lambda_\mathfrak{q}^r-\Lambda_{\tilde{\mathfrak{q}}}^r)(g)\tilde{g}d\sigma(x).
\]
Hence, in light of \eqref{u9} and \eqref{u10}, we deduce that
\begin{equation}\label{u11}
|\hat{\mathfrak{p}}(k)|\le Ch(\rho)+\mathfrak{D} \|g\|_{W^{2-1/r,r}(\Gamma)} \|\tilde{g}\|_{W^{2-1/r,r}(\Gamma)},\quad k\in \mathbb{R}^n\setminus\{0\},\; \rho \ge \rho_0,
\end{equation}
with $\mathfrak{p}=(\mathfrak{q}-\tilde{\mathfrak{q}})\chi_\Omega$ (in $H^s(\mathbb{R}^n)$).

On the other hand, inequalities \eqref{u9.1} and \eqref{u10.1} yield
\begin{align*}
&\|g\|_{W^{2-1/r,r}(\Gamma )} \le C_0 \|u\|_{W^{2,r}(\Omega )}\le C(1+M)^2e^{\kappa /h},
\\
&\|\tilde{g}\|_{W^{2-1/r,r}(\Gamma )} \le C_0 \|\tilde{u}\|_{W^{2,r}(\Omega )}\le C(1+M)^2e^{\kappa /h},
\end{align*}
where $C_0=C_0(\Omega ,r)>0$ is a constant

These estimates in \eqref{u11} yield
\[
C|\hat{\mathfrak{p}}(k)|\le h(\rho)+\mathfrak{D} (1+M)^4e^{\kappa /h(\rho )} , \quad k\in \mathbb{R}^n\setminus\{0\},\; \rho \ge \rho_0.
\]
That is we have
\[
C|\hat{\mathfrak{p}}(k)|\le 1/\rho+\mathfrak{D} (1+M)^4e^{\kappa (|k|/2+\rho)} ,\quad \quad k\in \mathbb{R}^n\setminus\{0\},\; \rho \ge \rho_0.
\]
Hence
\begin{equation}\label{u12}
C\int_{|k|\le \rho^{1/n}}|\hat{p}(k)|^2 dk\le 1/\rho+\mathfrak{D}(1+M)^4 e^{\kappa \rho} ,\quad  \rho \ge \rho_0.
\end{equation}

Moreover, 
\begin{align}
\int_{|k|\ge \rho^{1/n}} |\hat{\mathfrak{p}}(k )|^2dk &\le \rho^{-2s/n}\int_{|k|\ge h^{-\alpha}} |k^{2s} |\hat{\mathfrak{p}}(k)|^2dk \label{u13}
\\
&\le \rho^{-2s /n}\|p\|_{H^s (\mathbb{R}^n)}^2. \nonumber
\end{align}

Now inequalities \eqref{u12} and \eqref{u13} together with 
Planchel-Parseval's identity give
\begin{equation}\label{u14}
C\|\mathfrak{q}-\tilde{\mathfrak{q}}\|_{L^2(\Omega )} \le 1/\rho^\beta +\mathfrak{D} (1+M)^4 e^{\kappa \rho} ,\quad  \rho \ge \rho_0.
\end{equation}
with $\beta =\min \left(1/2 ,s/n\right)$ and $C=C(\Omega,r,s)>0$ is a constant.
\end{proof}

\section{Proof of the main result}\label{sec:5}

Before we proceed to the proof of Proposition \ref{proposition1.1}, we establish a lemma. To this end, let $\mathfrak{X}=H^2(\Omega)$ if $n\le 3$ and $\mathfrak{X}=W^{2,p}(\Omega )$ if $n\ge 4$, where $p$ is as in Theorem \ref{theorem1.2}.

\begin{lemma}\label{lemma5.1}
For any $a\in \mathscr{A}(\mathfrak{c},\alpha)$, the mapping 
\[
f\in \mathscr{X_0}_0\mapsto u_a(f)\in \mathfrak{X}
\]
is continuous.
\end{lemma}

\begin{proof}
Pick $f, h\in \mathscr{X}_0$. If $u=u_a(f+h)-u_a(f)$ then simple computations give that $u$ is the solution of the BVP
\[
\left\{
\begin{array}{ll}
-\Delta u+ru=0\quad &\mbox{in}\; \Omega ,
\\
u=h &\mbox{on}\; \Gamma,
\end{array}
\right.
\]
where 
\[
r(x)=\int_0^1a'(u_a(f)(x)+s[u_a(f+h)-u_a(f)](x))ds.
\]
We can then mimic the proof of Theorem \ref{theorem-existence} in order to find a constant $C>0$ independent on $h$ so that
\[
\|u\|_{\mathfrak{X}}\le \|h\|_{\mathscr{X}_0}.
\]
Thus the continuity of $f\in \mathscr{X_0}_0\mapsto u_a(f)\in \mathfrak{X}$ follows.
\end{proof}

\begin{proof}[Proof of Proposition \ref{proposition1.1}]
We give the proof in case (i). The proof for cases (ii) and (iii) is quite similar. 

Since the trace operator 
\[
w \in H^2(\Omega )\mapsto \partial_\nu w_\Gamma \in H^{1/2}(\Gamma )
\]
is bounded, it is sufficient to prove that 
\[
f\in H^{3/2}(\Gamma )\mapsto u_a(f)\in H^2(\Omega )
\]
is Fr\'echet differentiable.

Fix $N>0$ and let $f\in B_{H^{3/2}(\Gamma )}(N)$. Then, for any $h\in B_{H^{3/2}(\Gamma )}(1)$, we have $f+h\in B_{H^{3/2}(\Gamma )}(M)$, with $M=N+1$.

Let $v=v_{a,f}(h)$ and
\[
w=u_a(f+h)-u_a(f)-v. 
\]
It is then straightforward to verify that $w$ is the solution of the BVP
\[
\left\{
\begin{array}{ll}
-\Delta w =F\quad &\mbox{in}\; \Omega ,
\\
w=0 &\mbox{on}\; \Gamma,
\end{array}
\right.
\]
with
\begin{align*}
F(x)&=-a(u_a(f+h))(x)+a(u_a(f))(x)+ a'(u_a(f)(x))v(x)
\\
&=-\int_0^1\{a'(u_a(f)(x)+s[u_a(f+h)(x)-u_a(f)(x)])
\\
&\hskip 2cm \times [u_a(f+h)(x)-u_a(f)(x)]-a'(u_a(f)(x))v(x)\}ds
\\
&=-\int_0^1a'(u_a(f)(x)+s[u_a(f+h)(x)-u_a(f)(x)])w(x)ds
\\
&\hskip 1cm +\int_0^1\{a'(u_a(f)(x)+s[u_a(f+h)(x)-u_a(f)(x)])
\\
& \hskip 5.5cm -a'(u_a(f)(x))\}v(x)ds,
\end{align*}
where $v=v_{a,f}(h)$.
 
We decompose $F$ as $F=-\mathfrak{q}w+G$, where 
\begin{align*}
&\mathfrak{q}(x)=\int_0^1a'(u_a(f)(x)+s[u_a(f+h)(x)-u_a(f)(x)])ds,
\\
&G(x)=\int_0^1\{a'(u_a(f)(x)+s[u_a(f+h)(x)-u_a(f)(x)])
\\
&\hskip 5.5cm -a'(u_a(f)(x))\}v(x)ds.
\end{align*}
Under these new notations, we see that $w$ is the solution of the BVP
\[
\left\{
\begin{array}{ll}
-\Delta w+\mathfrak{q}w=G\quad &\mbox{in}\; \Omega ,
\\
w=0 &\mbox{on}\; \Gamma.
\end{array}
\right.
\]
According to Corollary \ref{corollary2.2}, we have
\[
\|u_a(f+h)\|_{L^\infty (\Omega)}\le C,
\]
where $C=C(\Omega,\mathfrak{c},\alpha,M)>0$ is a constant.

Using \eqref{1.5} for estimating the integrand of the definition of
$\mathfrak{q}(x)$ and applying triangle's inequality, we obtain
\[
\|\mathfrak{q}\|_{L^\infty (\Omega )}\le \|a'(0)\|_{L^\infty (\Omega )}+\varkappa_C C:= \tilde{C}.
\]
We obtain from  the usual a priori $H^2$-estimate (e.g., 
\cite[Sections 8.5 and 8.6]{RR}) that
\[
\|w\|_{H^2(\Omega )}\le \hat{C}\|G\|_{L^2(\Omega )},
\]
where $\hat{C}=\hat{C}(\Omega,a,\mathfrak{c},\alpha,M)$ is a constant.
But
\begin{align*}
\|G\|_{L^2(\Omega )}&\le \varkappa_C\|u_a(f+h)-u_a(f)\|_{L^2(\Omega )}\|v\|_{L^\infty (\Omega )}
\\
&\le \varkappa_Cc_\Omega \|u_a(f+h)-u_a(f)\|_{L^2(\Omega )}\|v\|_{H^2(\Omega )}.
\end{align*}

Therefore, again from $H^2$ a priori estimates for $v$, we have
\[
\|w\|_{H^2(\Omega )}\le \hat{C}\varkappa_Cc_\Omega \|u_a(f+h)-u_a(f)\|_{L^2(\Omega )}\|h\|_{H^{3/2}(\Gamma )}.
\]
Now we complete the proof of the differentiability of $f\mapsto u_a(f)$ by 
using 
that, according to Lemma \ref{lemma5.1}, the mapping
\[
f\in H^{3/2}(\Gamma )\mapsto u_a(f)\in H^2(\Omega )
\]
is continuous. 
\end{proof}

Define
\[
\mathfrak{q}_{a,f}:=a'\circ u_a(f)  .
\]

 In order to apply the results of the preceding section we need to extend $\Lambda_{\mathfrak{q}_{a,f}}$ to complex-valued functions from $H^{3/2}(\Gamma)$. As $\mathfrak{q}_{a,f}$ is real-valued, this extension is obviously given by
 \[
 \Lambda_{\mathfrak{q}_{a,f}}(f+ig)=\Lambda_{\mathfrak{q}_{a,f}}(f)+i\Lambda_{\mathfrak{q}_{a,f}}(g),\quad f,g\in H^{3/2}(\Gamma)\; \mbox{real-valued}.
 \]
 It is then useful to observe that this extension is entirely determined by it restriction to real-valued functions from $H^{3/2}(\Gamma)$.

Proceeding as in the proof of Proposition \ref{proposition1.1}, we prove the following result.

\begin{lemma}\label{lemma2.4}
Let $\beta$ be as in Theorem \ref{theorem1.3}. Under the assumptions and the notations of Proposition \ref{proposition1.1}, we have $\mathfrak{q}_{a,f}\in C^{0,\beta}(\overline{\Omega})$ and
\begin{equation}\label{2.18}
\|\mathfrak{q}_{a,f}\|_{C^{0,\beta}(\overline{\Omega})}\le C.
\end{equation}
Here the constant $C>0$  is so that $C=C(\Omega,\mathfrak{c},\alpha,M)$ if $n=2$ or $n=3$ ; $C=C(\Omega,\mathfrak{c},\alpha,p,r)$ if $n=4$ ; $C=C(\Omega,\mathfrak{c},\alpha,M,p)$ if $n>4$.
\end{lemma}

Following \cite[Definition 1.3.2.1, page 16]{Gr}, the space
$H^t(\Omega )$, $0<t<1$, consists of functions $w\in L^2(\Omega)$ 
satisfying
\[
\int_\Omega \int_\Omega \frac{|w(x)-w(y)|^2}{|x-y|^{n+2t}}dxdy<\infty .
\]
Let $0<t<\theta\le 1$ and $w\in C^{0,\theta}(\overline{\Omega})$. Then
\begin{equation}\label{em1}
\frac{|w(x)-w(y)|^2}{|x-y|^{n+2t}}\le \frac{[w]_\theta^2 }{|x-y|^{n+2t-2\theta}},
\end{equation}
where 
\[
[w]_\theta=\sup\{|w(x)-w(y)|/|x-y|^\theta;\; x,y\in \overline{\Omega},\; x\ne y\}.
\]
On the other hand, for any $\epsilon >0$, we have
\begin{equation}\label{em2}
\int_{B(x,\epsilon)}\frac{1}{|x-y|^{n+2t-2\theta}}dy=\int_{\mathbb{S}^{n-1}}d\omega\int_0^\epsilon \frac{1}{t^{2t-2\theta +1}}dt.
\end{equation}
Consequently, since the integral in \eqref{em2} is convergent
by $2t-2\theta + 1 < 1$, in terms of inequality \eqref{em1}
we can directly see that $C^{0,\theta}(\overline{\Omega})$ is continuously 
embedded in $H^t(\Omega )$. 
Hence an immediate consequence of the previous lemma is the following corollary.

\begin{corollary}\label{corollary2.4}
Let $\beta$ be as in Theorem \ref{theorem1.3}. Under the assumptions and the notations of Proposition \ref{proposition1.1}, we have $\mathfrak{q}_{a,f}\in C^{0,\beta}(\overline{\Omega})\cap H^s(\Omega )$ for $0<s<\min (1/2,\beta)$ and
\begin{equation}\label{2.19}
\|\mathfrak{q}_{a,f}\|_{C^{0,\beta}(\overline{\Omega})}+\|\mathfrak{q}_{a,f}\|_{H^s(\Omega )}\le C,
\end{equation}
where the constant $C>0$ can be described as 
$C=C(\Omega,\mathfrak{c},\alpha,M)$ if $n=2$ or $n=3$; $C=C(\Omega,\mathfrak{c},\alpha,p,r)$ if $n=4$ ; $C=C(\Omega,\mathfrak{c},\alpha,M,p)$ if $n>4$.
\end{corollary}

\begin{proof}[Proof of Theorem \ref{theorem1.3}] 
In this proof $C>0$, $\rho_0>0$ and $\kappa>0$ are generic constants only 
depending : on $(\Omega,\mathfrak{c},\alpha, M,s)$ if $n=2,3$,
$(\Omega,\mathfrak{c},\alpha, M,s,p,r)$ if $n=4$, 
$(\Omega,\mathfrak{c},\alpha, M,s,p)$ if $n>4$. The constants $p$ and $r$ are the same as in Theorem \ref{theorem1.3}.

Using \eqref{2.19} for both $a$ and $\tilde{a}$, we obtain by applying Theorem \ref{theorem-stab}
\begin{equation}\label{mt1}
C\|\mathfrak{q}_{a,f}-\tilde{\mathfrak{q}}_{\tilde{a},f}\|_{L^2(\Omega )} \le 1/\rho^\gamma +\mathfrak{D} (f) e^{\kappa \rho} ,\quad  \rho \ge \rho_0,
\end{equation}

where $\gamma =\min (1/2 ,s/n)$ and 
\[
\mathfrak{D} (f)=\|\Lambda'_a(f)-\Lambda'_{\tilde{a}}(f)\|_{\mathscr{Y}}.
\]
Now the interpolation inequality in \cite[Lemma B.1]{CK} gives
\begin{equation}\label{mt2}
\|\mathfrak{q}_{a,f}-\tilde{\mathfrak{q}}_{\tilde{a},f}\|_{C(\overline{\Omega })}\le C_0\|\mathfrak{q}_{a,f}-\tilde{\mathfrak{q}}_{\tilde{a},f}\|_{C^{0,\beta}(\Omega )}^{n/(n+2\beta)}\|\mathfrak{q}_{a,f}-\tilde{\mathfrak{q}}_{\tilde{a},f}\|_{L^2(\Omega )}^{2\beta /(n+2\beta)}.
\end{equation}

Inequalities \eqref{mt2} and \eqref{2.19} both for $a$ and $\tilde{a}$ imply
\begin{equation}\label{mt3}
\|\mathfrak{q}_{a,f}-\tilde{\mathfrak{q}}_{\tilde{a},f}\|_{C(\overline{\Omega })}\le C\|\mathfrak{q}_{a,f}-\tilde{\mathfrak{q}}_{\tilde{a},f}\|_{L^2(\Omega )}^{2\beta /(n+2\beta)}.
\end{equation}
We find by putting \eqref{mt3} in \eqref{mt1}
\[
C\|\mathfrak{q}_{a,f}-\tilde{\mathfrak{q}}_{\tilde{a},f}\|_{C(\overline{\Omega })}^{(n+2\beta)/(2\beta)} \le 1/\rho^\gamma +\mathfrak{D} (f) e^{\kappa \rho} ,\quad  \rho \ge \rho_0.
\]

Using this inequality with $f=\lambda $ such that $|\lambda |\le M$,
we have
\begin{equation}\label{mt4}
C\left[\max_{|\lambda |\le M}|a'(\lambda )-\tilde{a}'(\lambda )|\right]^{(n+2\beta)/(2\beta)}\le 1/\rho^\gamma +\mathfrak{D}_Me^{\kappa \rho} ,\quad  \rho \ge \rho_0,
\end{equation}
with
\[
\mathfrak{D}_M=\sup_{\|f\|_{\mathscr{X}_0}\le \sqrt{|\Gamma|}M}\mathfrak{D}(f).
\]
Since $a(0)=\tilde{a}(0)$, we have
\[
\max_{|\lambda |\le M}|a(\lambda )-\tilde{a}(\lambda )|\le \max_{|\lambda |\le M}|a'(\lambda )-\tilde{a}'(\lambda )|.
\]
This in \eqref{mt4} yields
\begin{equation}\label{mt5}
C\left[\max_{|\lambda |\le M}|a(\lambda )-\tilde{a}(\lambda )|\right]^{(n+2\beta)/(2\beta)}\le 1/\rho^\gamma +\mathfrak{D}_Me^{\kappa \rho} ,\quad  \rho \ge \rho_0.
\end{equation}

For completing the proof we choose $\rho\ge \rho_0$ which makes the 
right-hand side nearly minimum. 
Let $\tau=\rho_0e^{\kappa _0}$. 
Since the mapping $\rho \in [0,\infty )\mapsto \rho^\gamma e^{\kappa \rho}$ is 
increasing, we see that if $\mathfrak{D}_M< \mu= \min (1,\tau^{-1})$,
then we can find $\rho_1\ge \rho_0$ so that $1/\rho_1^\gamma =\mathfrak{D}_Me^{\kappa \rho_1}$. Therefore, by taking $\rho=\rho_1$ in \eqref{mt5}, we find
\[
C\left[\max_{|\lambda |\le M}|a(\lambda )-\tilde{a}(\lambda )|\right]^{(n+2\beta)/(2\beta)}\le 1/\rho_1^\gamma. 
\]
Now elementary computations show that $\rho_1^{-1}\le (\kappa +\gamma)|\ln \mathfrak{D}_M|^{-1}$. Hence
\[
\max_{|\lambda |\le M}|a(\lambda )-\tilde{a}(\lambda )|\le C |\ln \mathfrak{D}_M|^{-(2\beta \gamma /(n+2\beta )} .
\]
When $\mathfrak{D}_M\ge \mu$, we have 
\[
\max_{|\lambda |\le M}|a(\lambda )-\tilde{a}(\lambda )|\le C \le C\mu^{-1}\mathfrak{D}_M.
\]
We complete the proof by putting together the last two inequalities.
\end{proof}


\begin{thebibliography}{99}

\bibitem{CDR} P. Caro, D. Dos Santos Ferreira and A. Ruiz,
\newblock Stability estimates for the Calder\'on problem with partial data,
\newblock J. Different. Equat. 260 (3) (2016) 2457-2489.

\bibitem{CaK} P. Caro and Y. Kian,
\newblock Determination of convection terms and quasi-linearities appearing in diffusion equations,
\newblock arXiv:1812.08495.

\bibitem{CLOP} X. Chen, M. Lassas, L. Oksanen, and G. Paternain, 
\newblock Detection of Hermitian connections in wave equations with cubic non-linearity,
\newblock arXiv:1902.05711.

\bibitem{Chou} M. Choulli, 
\newblock Inverse problems for Schr\"odinger equations with unbounded potentials, 
\newblock arXiv:1909.11133.


\bibitem{CK} M. Choulli and Y. Kian, Logarithmic stability in determining the time-dependent zero order coefficient in a parabolic equation from a partial Dirichlet-to-Neumann map. Application to the determination of a nonlinear term, J. Math. Pures Appl. 114 (2018) 235-261.

\bibitem{COY} M. Choulli, E. M. Ouhabaz and M. Yamamoto,
\newblock Stable determination of a semilinear term in a parabolic equation,
\newblock Commun. Pure Appl. Anal. 5 (3) (2006) 447-462.


\bibitem{GT} D. Gilbarg and N. S. Trudinger,
\newblock Elliptic partial differential equations of second order, Springer, Berlin, 1998.

\bibitem{Gr} P. Grisvard,
\newblock Elliptic problems in nonsmooth domains, Pitman Publishing Inc., 1985.

\bibitem{HUZ}P. Hintz, G. Uhlmann and J. Zhai,
\newblock An inverse boundary value problem for a semilinear wave equation
on Lorentzian manifolds, 
\newblock arXiv:2005.10447.

\bibitem{Is} V. Isakov, 
\newblock On uniqueness in inverse problems for semilinear parabolic equations,
\newblock Arch. Rational Mech. and Anal.  124 (1993), 1-12.

\bibitem{IN} V. Isakov and A. Nachman,
\newblock Global Uniqueness for a two-dimensional elliptic inverse problem,
\newblock Trans. Amer. Math. Soc. 347 (1995), 3375-3391.


\bibitem{IS} V. Isakov and J. Sylvester,
\newblock Global uniqueness for a semilinear elliptic inverse problem,
\newblock Commun. Pure Appl. Math. 47 (1994), 1403-1410.

\bibitem{IY2013} O. Yu Imanuvilov and M. Yamamoto, 
\newblock Unique determination of potentials and semilinear terms of semilinear elliptic equations from partial Cauchy data, 
\newblock J. Inverse Ill-Posed Probl. 21 (2013), 85-108.

\bibitem{KN} H. Kang and G. Nakamura,
\newblock  Identification of nonlinearity in a conductivity equation via the Dirichlet-to-Neumann map, 
\newblock Inverse Problems, 18 (4) (2002), 1079 .

\bibitem{KU} K. Krupchyk and G.Uhlmann,
\newblock A remark on partial data inverse problems for semilinear elliptic equations, \newblock arXiv:1905.01561.

\bibitem{KLU} Y. Kurylev,  M. Lassas,  and G. Uhlmann,
\newblock  Inverse problems for Lorentzian manifolds and non-linear hyperbolic equations,
\newblock Inventiones mathematicae, 212 (3) (2018), 781-857.

\bibitem{LLLS1} M. Lassas, T. Liimatainen, Y.-H.  Lin, and M. Salo,
\newblock Inverse problems for elliptic equations with power type nonlinearities,
\newblock arXiv:1903.12562.

\bibitem{LLLS} M. Lassas, T. Liimatainen, Y.-H.  Lin, and M. Salo,
\newblock Partial data inverse problems and simultaneous recovery of boundary and coefficients for semilinear elliptic equations,
\newblock arXiv:1905.02764.

\bibitem{LM} J-L. Lions and E. Magenes, 
\newblock Non homogeneous boundary value problems and applications, 
\newblock Vol. I, Springer, Berlin, 1972.

\bibitem{MU} C. Munoz and G. Uhlmann,
\newblock The Calder\'on problem for quasilinear elliptic equations,
\newblock arXiv:1806.09586.

\bibitem{RR} M. Renardy and R.C. Rogers, 
\newblock An introduction to partial differential equations,
\newblock Springer, NewYork, 1993.

\bibitem{Su} Z. Sun,
\newblock  On a quasilinear inverse boundary value problem,
\newblock Math. Z., 221 (2) (1996), 293-305.

\bibitem{SU} Z. Sun and G.Uhlmann,
\newblock Inverse problems in quasilinear anisotropic media,
\newblock American J. Math., 119 (4) (1997), 771-797.

\bibitem{WZ}  Y. Wang and T. Zhou, 
\newblock Inverse problems for quadratic derivative nonlinear wave equations,
\newblock Commun. Part. Different. Equat., to appear, 2019.

\end{thebibliography}
\end{document}